\documentclass[11pt]{amsart}
\usepackage{latexsym}
\usepackage{amstext}
\usepackage{amsmath}
\usepackage{amssymb}
\usepackage{amsthm}
\usepackage{url}
\usepackage{enumitem}

\newtheorem{theorem}{Theorem}[section]
\newtheorem{lemma}[theorem]{Lemma}
\newtheorem{proposition}[theorem]{Proposition}

\newtheorem{theorema}[theorem]{Theorem}
\newtheorem{theoremb}{Theorem}[section]
\newtheorem{corollarya}[theorem]{Corollary}
\newtheorem{corollaryb}[theoremb]{Corollary}

\theoremstyle{definition}
\newtheorem{definition}[theorem]{Definition}
\newtheorem{remark}[theorem]{Remark}
\newtheorem{example}[theorem]{Example}

\def \co {\mathcal{O}}
\def \Ocal {\mathcal{O}}

\def\QQ{\mathbb Q}

\def\k{\mathbf k}

\DeclareMathOperator{\dv}{div}

\DeclareMathOperator{\Pic}{Pic}
\DeclareMathOperator{\chr}{char}

\renewcommand{\theenumi}{(\alph{enumi})}

\begin{document}
\title[On non-Archimedean curves omitting few components]{On non-Archimedean curves omitting few components and their arithmetic analogues}

\author[A. Levin]{Aaron Levin}
\address{Department of Mathematics\\Michigan State University\\East Lansing, MI 48824}
\email{adlevin@math.msu.edu}
\author[J. T.-Y. Wang]{Julie Tzu-Yueh Wang}
\address{Institute of Mathematics\\Academia Sinica\\No.\ 1, Sec.\ 4, Roosevelt Road\\
Taipei 10617\\TAIWAN}
\email{jwang@math.sinica.edu.tw}
\date{}
\begin{abstract}
Let $\k$ be an algebraically closed field complete with respect to a non-Archimedean absolute value of arbitrary characteristic.
Let $D_1,\ldots, D_n$ be effective nef divisors intersecting transversally in an $n$-dimensional nonsingular projective variety  $X$.   
We study the degeneracy of non-Archimedean analytic maps from $\k$ into $X\setminus \cup_{i=1}^nD_i$  under various geometric conditions.  When $X$ is a rational ruled surface and $D_1$ and $D_2$ are ample, we obtain a necessary and sufficient condition such that
there is no  non-Archimedean analytic map from $\k$ into $X\setminus D_1 \cup D_2$.
Using the dictionary between non-Archimedean Nevanlinna theory and Diophantine approximation that originated in \cite{ALW}, we also study arithmetic analogues of these problems, establishing results on integral points on these varieties over $\mathbb{Z}$ or the ring of integers of an imaginary quadratic field.
\end{abstract}
\thanks{2010\ {\it Mathematics Subject Classification}: Primary 11J97; Secondary 32P05, 32H25}
\keywords{non-Archimedean Picard theorem, non-Archimedean analytic curves, integral points}
\thanks{The first named author was supported in part by NSF grants DMS-1102563 and DMS-1352407.  The second named author was supported in part by Taiwan's MOST grant 103-2115-M-001-002-MY3.}
\normalsize
\baselineskip=17pt

\maketitle

 \section{Introduction}

Let $\k$ be an algebraically closed field complete with respect to a non-Archimedean absolute value of arbitrary characteristic.  Our primary object of study is the degeneracy of non-Archimedean analytic maps from $\k$ to an $n$-dimensional projective variety $X$ omitting an effective divisor with at least $n$ irreducible components.  As argued in \cite{ALW}, results on non-Archimedean analytic curves in this context should have arithmetic counterparts in Diophantine geometry; they should correspond to results on integral points over $\mathbb{Z}$ or the ring of integers of an imaginary quadratic field.  Thus, a second objective is to prove an appropriate arithmetic analogue of all of our results on non-Archimedean analytic curves, further illustrating and justifying the correspondence proposed in \cite{ALW}.  Before discussing our main results, we briefly recall some aspects of the correspondence in \cite{ALW} as well as several examples of parallel non-Archimedean and arithmetic results.

To begin, we recall the connection between Nevanlinna theory, the quantitative theory that grew out of Picard's theorem, and Diophantine approximation, a quantitative theory behind many results on rational and integral points on varieties.  Originating in the work of Osgood, Vojta, and Lang, it has been observed that there is a striking correspondence between many statements in Nevanlinna theory and statements in Diophantine approximation.  A detailed ``dictionary" between the two subjects has been constructed by Vojta \cite{Vo}.  Qualitatively, in the simplest case, holomorphic curves in a variety $X$ should correspond to infinite sets of integral points on $X$ (viewing $X$ both as a variety over a number field and as a complex analytic space).  Taking $X=\mathbb{G}_m=\mathbb{P}^1\setminus \{0,\infty\}$, this implies that a non-constant entire function without zeros is analogous to an infinite set of units in some ring of integers. 

In contrast to the situation in complex analysis, where the exponential function is an entire function without zeros, it is an easy fact that a non-Archimedean entire function without zeros must be constant.  In view of the aforementioned analogies, this suggests that to obtain a Diophantine analogue of non-Archimedean analytic curves, we should consider only rings of integers with a finite unit group, i.e., $\mathbb{Z}$ or the ring of integers of an imaginary quadratic field.  This observation was the starting point for work in \cite{ALW}, where more generally it was argued that, at least for certain classes of varieties, a non-constant non-Archimedean analytic map into a variety $X$ should correspond to an infinite set of $\co_k$-integral points on $X$, where $\co_k=\mathbb{Z}$ or the ring of integers of an imaginary quadratic field.  For completeness, we also mention that this correspondence can frequently be made quantitative (see \cite{ALW}), resulting in parallel statements in non-Archimedean Nevanlinna theory and Diophantine approximation.

Determining the precise class of varieties under which a correspondence should hold was left open in \cite{ALW}.  At the least, it seems necessary that the varieties $X$ be affine (or close to affine) and that, in the arithmetic case, the varieties satisfy a rationality condition on the components at infinity:
 
\begin{enumerate}[label= {($\ast$)}]
\item \label{cond1} There exists a projective closure $\tilde{X}$ of $X$ nonsingular at every point in $\tilde{X}\setminus X$ 
and such that every (geometric) irreducible component of $\tilde{X}\setminus X$ is defined over $k$,        \end{enumerate}
where $k=\mathbb{Q}$ or an imaginary quadratic field.  We now illustrate the correspondence by recalling several examples of parallel non-Archimedean and arithmetic results.

For curves, the condition \ref{cond1} yields a sufficient hypothesis under which our correspondence holds (see Section \ref{prelim} for the definitions):
\begin{theorem}
Let $k=\mathbb{Q}$ or an imaginary quadratic field and suppose that $\chr \k=0$.  If $X$ is an affine curve over $k$ satisfying \ref{cond1} then $X$ contains an infinite set of $\co_k$-integral points if and only if there exists a non-constant analytic map $f:\k\to X$ if and only if $X$ is rational with a single point at infinity.
\end{theorem}
The last equivalence holds in arbitrary characteristic (we only assumed $\chr \k=0$ to ensure that $X$ makes sense over both $k$ and $\k$).  This follows easily from Siegel's theorem and Berkovich's non-Archimedean analogue of Picard's theorem.  As a special case, we have the following version of Berkovich's Picard theorem.
\renewcommand{\thetheorem}{\arabic{section}.\arabic{theorem}A}
\setcounter{theoremb}{\value{theorem}}
\begin{theorema}
\label{Pic}
Any analytic map from $\k$ to a projective curve omitting two points must be constant.
\end{theorema}

This corresponds to the following theorem on integral points.

\begin{theoremb}
\label{Sieg}
Let $k=\mathbb{Q}$ or an imaginary quadratic field and let $C$ be an affine curve over $k$ with at least two $k$-rational points at infinity.  Then any set of $\co_k$-integral points on $C$ is finite.
\end{theoremb}

A generalization of Theorem \ref{Pic} to higher dimensions was obtained by Lin and Wang \cite{LW}, and a refinement was given by An, Levin and Wang in  \cite{ALW} as follows.
 \begin{theorem} 
\label{ALW}
Let $X$ be a nonsingular projective variety over $\k$. Let $D_1,\ldots,D_m$ be
 effective divisors on $X$ with empty intersection.  Let $D=\sum_{i=1}^mD_i$.  
\begin{enumerate}[label={\rm \theenumi}]  
\item  If $\kappa(D_i)>0$ for all $i$, then the image of an analytic map 
$f:\k\to X\setminus D$ is contained in a proper subvariety of $X$.
\item  If $D_i$ is big for all $i$, then there exists a proper Zariski-closed subset $Z\subset X$ such that the image of any non-constant analytic map 
$f:\k\to X\setminus D$ is contained in $Z$.
\item  If $D_i$ is ample for all $i$, then there is no non-constant analytic map from $\k$ to $X\setminus D$.
\end{enumerate}
\end{theorem}
The arithmetic analogue of Theorem \ref{ALW} is implicit in the proof of a more general result proved in \cite{Run} using a higher-dimensional version of ``Runge's method".
\begin{theoremb}[Levin]
\label{Levin}
Let $k=\mathbb{Q}$ or an imaginary quadratic field.  Let $X$ be a nonsingular projective variety over $k$.  Let $D_1,\ldots,D_m$ be
 effective divisors on $X$, defined over $k$, with empty intersection.  Let $D=\sum_{i=1}^mD_i$.
\begin{enumerate}[label={\rm \theenumi}] 
\item   If $\kappa(D_i)>0$ for all $i$, then any set $R$ of $\co_{k}$-integral points on $X\setminus D$ is contained in a proper Zariski-closed subset of $X$.
\item If $D_i$ is big for all $i$, then there exists a proper Zariski-closed subset $Z\subset X$ such that for any set $R$ of $\co_{k}$-integral points on $X\setminus D$, the set $R\setminus Z$ is finite.
\item  If $D_i$ is ample for all $i$, then all sets $R$ of $\co_{k}$-integral points on $X\setminus D$ are finite.\label{Levinc}
\end{enumerate}
\end{theoremb}
Furthermore, as shown in \cite{Run}, in each of the above cases, the integral points can be effectively computed.

We note that if 
 $D_1,\ldots,D_m$ are effective divisors in general position on $X$, then they have empty intersection if $m\ge \dim X+1$.
 Here, a collection of effective divisors $D_i$ on a projective variety $X$ of dimension $n$ is said to be {\it in general position} if for each 
 $1\le k\le n+1$ and each choice of indices $i_1<\ldots<i_k$, each irreducible component of $D_{i_1}\cap\ldots\cap D_{i_k}$ has codimension at least $k$ in $X$;  in particular, this intersection is empty when $k=n+1$.
 
In view of Theorems \ref{ALW} and \ref{Levin}, a next natural case to study is varieties of the form $X\setminus \cup_{i=1}^nD_i$, where $n=\dim X$.  The non-Archimedean case was studied for projective space by An, Wang, and Wong \cite{AWW} (cf.\ \cite{ALW} for a necessary and sufficient statement), and the arithmetic analogue was established in \cite{ALW}.

\begin{theorem}[An, Wang, Wong]
\label{AWW2}
Let $D_1,\ldots,D_n$ be nonsingular hypersurfaces in $\mathbb{P}^n$ intersecting transversally.
Then there is no non-constant analytic map from $\k$ to $\mathbb{P}^n\setminus \cup_{i=1}^nD_i$ if $\deg D_i\geq 2$ for each $1\leq i \leq n$.
\end{theorem}

\begin{theoremb}[An, Levin, Wang]
Let $k=\mathbb{Q}$ or an imaginary quadratic field.  Let $D_1,\ldots, D_n$ be nonsingular hypersurfaces defined over $k$ in $\mathbb{P}^n$ intersecting transversally.  Suppose that every point in the intersection $\cap_{i=1}^nD_i$ is $k$-rational.  Then any set of $\co_k$-integral points on $\mathbb{P}^n \setminus \cup_{i=1}^nD_i$ is finite.
\end{theoremb}

More generally, when $k=\mathbb{Q}$ or an imaginary quadratic field and $D_1$ and $D_2$ are curves over $k$ intersecting transversally,  a complete characterization  of finiteness of $\co_k$-integral points on $\mathbb{P}^2 \setminus \{D_1 \cup D_2\}$ was given in \cite{ALW}.

We now state our main results.  Under various geometric conditions, we prove the degeneracy of analytic maps $f:\k\to X\setminus \cup_{i=1}^nD_i$.  We also prove analogous results for $\co_k$-integral points on $X\setminus \cup_{i=1}^nD_i$ under the condition \ref{cond1} and the condition
\begin{enumerate}[label= {($\ast$\!$\ast$)}]
\item\label{cond2}Every point in the intersection $\cap_{i=1}^nD_i$ is $k$-rational.
\end{enumerate}
 Notation and definitions will be given in the next section.
\renewcommand{\thetheorem}{\arabic{section}.\arabic{theorem}A}
\setcounter{theoremb}{\value{theorem}} 
\begin{theorema}
\label{pxn}
Let $D_1,\ldots, D_n$ be effective nef divisors intersecting transversally in an $n$-dimensional nonsingular projective variety  $X$ over $\k$.   Let $K_X$ denote the canonical divisor on $X$.
\begin{enumerate}[label={\rm \theenumi}] 
\item  Assume that  either $D_i^n>1$ or  that $D_i^n=1$ and  $K_X.D_i^{n-1}<1-n$ for each $1\le i\le n$.  Then the image of an analytic map 
$f:\k\to X\setminus \cup_{i=1}^nD_i$ is contained in a proper subvariety of $X$. 
\item  If $D_i^n>1$ for all $i$, then there exists a proper Zariski-closed subset $Z\subset X$ such that the image of any non-constant analytic map $f:\k\to X\setminus  \cup_{i=1}^nD_i$ is contained in $Z$.
\end{enumerate}
\end{theorema}

\begin{theoremb} 
\label{pxnA}
Let $k=\QQ$ or an imaginary quadratic field.  Let $X$ be an $n$-dimensional nonsingular projective variety over $k$.  Let $D_1,\ldots, D_n$ be effective nef divisors on $X$, all defined over $k$,  intersecting transversally in $X$.  Let $K_X$ denote the canonical divisor on $X$.  Suppose that every point in the intersection $\cap_{i=1}^nD_i$ is $k$-rational.

\begin{enumerate}[label={\rm \theenumi}] 
\item   Assume that  either $D_i^n>1$ or  that $D_i^n=1$ and  $K_X.D_i^{n-1}<1-n$ for each $1\le i\le n$.  Then any set $R$ of $\Ocal_k$-integral points on $X\setminus \cup_{i=1}^nD_i$ is contained in a proper Zariski-closed subset of $X$. 
\item  If $D_i^n>1$ for all $i$, then there exists a proper Zariski-closed subset $Z\subset X$ such that for any set $R$ of $\Ocal_k$-integral points on $X\setminus \cup_{i=1}^nD_i$, the set $R\setminus Z$ is finite.
\end{enumerate}
\end{theoremb}

Combined with an appropriate version of the Hodge index theorem (Theorem \ref{HI}), this yields the following corollaries.

\begin{corollarya}\label{cor}
Let $X$ be an $n$-dimensional nonsingular projective variety over $\k$.  
Suppose that $-K_X$ is nef and 
\begin{align*}
(-K_X)^n>(n-1)^n.
\end{align*}
Let $D_1,\ldots, D_n$ be effective nef and big divisors on $X$ intersecting transversally.  Then  the image of an analytic map 
$f:\k\to X\setminus \cup_{i=1}^nD_i$ is contained in a proper subvariety of $X$.  
\end{corollarya}
\begin{corollaryb} \label{corb}
Let $k=\QQ$ or an imaginary quadratic field.  Let $X$ be an $n$-dimensional nonsingular projective variety over $k$.   Suppose that $-K_X$ is nef and 
\begin{align*}
(-K_X)^n>(n-1)^n.
\end{align*}
Let $D_1,\ldots, D_n$ be effective nef and big divisors on $X$, all defined over $k$,  intersecting transversally in $X$.  Suppose that every point in the intersection $\cap_{i=1}^nD_i$ is $k$-rational.  Then any set $R$ of $\Ocal_k$-integral points on $X\setminus \cup_{i=1}^nD_i$ is contained in a proper Zariski-closed subset of $X$. 
\end{corollaryb}

The surfaces $X$ satisfying the hypotheses of Corollary \ref{cor} can be completely classified as follows. 
(cf.  \cite[Chapter III, Exercise 3.8]{Kol} ) 
\renewcommand{\thetheorem}{\arabic{section}.\arabic{theorem}}
\begin{remark}
Let $X$ be a smooth projective surface over an algebraically closed field.  Then 
$-K_X$ is nef and $(-K_X)^2>1$ if and only if $X$ is one of the following:
\begin{enumerate}
\item $X\cong \mathbf {P}({\mathcal O}_{\mathbb P^1}\oplus{\mathcal O}_{\mathbb P^1}(2))$ over ${\mathbb P^1}$;  
\item $X\cong {\mathbb P^2}$ or  $X\cong {\mathbb P^1}\times {\mathbb P^1}$;
\item $X$  is obtained from ${\mathbb P^2}$ by successively blowing up at most 7 points.
\end{enumerate}
\end{remark}

In contrast to the case of projective space (Theorem \ref{AWW2}) and the situation for degeneracy (Theorem \ref{pxn}), on general projective varieties, even under a transversality assumption,  there is no condition on the self-intersection numbers of the divisors or the degrees of the divisors (in some fixed projective embedding) sufficient to ensure that there are no non-constant analytic maps $f:\k\to X\setminus (D_1\cup\cdots \cup D_n)$.  A similar statement holds for integral points, but we give an example only in the non-Archimedean setting.

\begin{example}
\label{hyperbex}
Let $m$ and $n$ be positive integers.  Let $X=\mathbb{P}^1\times \mathbb{P}^1$ and let $D_1$ and $D_2$ be effective divisors of type $(1,m)$ and $(1,n)$, respectively.  Let $D=D_1+D_2$.  Let $P$ be a point in the intersection of the supports of $D_1$ and $D_2$ and let $L$ be the line on $X$ of type $(0,1)$ through $P$.  Then since $L\setminus D=L\setminus \{P\}\cong \mathbb{A}^1$, there exists a non-constant analytic map $f:\k\to X\setminus D$.  Note that $D_1$ and $D_2$ are very ample and $D_1^2=2m$, $D_2^2=2n$.  Thus, there is no condition on the self-intersection numbers of (very) ample divisors $D_1$ and $D_2$ on $X$ that is sufficient to guarantee that there are no non-constant analytic maps from $\k$ to $X\setminus (D_1\cup D_2)$.  Note also that for any embedding $X\subset \mathbb{P}^N$, the divisors $D_i$ may have arbitrarily large degree.
\end{example}

In the case of rational ruled surfaces, we are able to completely classify analytic maps $f:\k \to X\setminus D_1\cup D_2$ when $D_1$ and $D_2$ are ample effective divisors intersecting transversally.  We also prove the arithmetic analogue, again under the conditions \ref{cond1} and \ref{cond2}.

Recall that if $X$ is a rational ruled surface, then $X\cong \mathbf{P}(\mathcal{E})$, where $\mathcal{E}=\co\oplus \co(-e)$ over $\mathbb{P}^1$ and $e\geq 0$ is uniquely determined.  We let $F$ denote a fiber on $X$ and let $C_0$ denote a section of $X$ such that $\co(C_0)\cong \co_{\mathbf{P}(\mathcal{E})}(1)$. 
\renewcommand{\thetheorem}{\arabic{section}.\arabic{theorem}A}
\setcounter{theoremb}{\value{theorem}}
\begin{theorem}
\label{rationalsurface}
Let $X$ be a rational ruled surface over $\k$.  Let $D_1$ and $D_2$ be effective divisors intersecting transversally in $X$.  
\begin{enumerate}[label={\rm \theenumi}]
\item
If $D_1$ and $D_2$ are  big,
then  the image of an analytic map 
$f:\k\to X\setminus D_1\cup D_2$ is contained in a proper subvariety of $X$.
\item
Suppose that $D_1$ and $D_2$ are ample.  The image of an analytic map 
$f:\k\to X\setminus D_1\cup D_2$ is contained in either a fiber or a section $C$ with $C\sim C_0$.  There is no non-constant analytic map from $\k$ to $X\setminus D_1\cup D_2$ if and only if every fiber and every section $C$, $C\sim C_0$, intersects $D_1\cup D_2$ in more than one point.  In particular, this holds if $D_i.F\ge 2$ and $D_i.C_0\ge 2$ for $i=1,2$.\label{rsb}
\end{enumerate}
\end{theorem}
\begin{theoremb} 
\label{rationalsurfaceb}
Let $k=\QQ$ or an imaginary quadratic field.
Let $\pi:X\to \mathbb{P}^1$ be a rational ruled surface over $k$.  Let $D_1$ and $D_2$ be effective divisors intersecting transversally in $X$.   Suppose that  all the irreducible components of $D_1$ and $D_2$  are defined over $k$ and that all of the points in the intersection $D_1\cap D_2$ are $k$-rational.
\begin{enumerate}[label={\rm \theenumi}] 
\item 
If $D_1$ and $D_2$ are big,
then any set $R$ of $\Ocal_k$-integral points on $X\setminus D_1\cup D_2$ is contained in a proper Zariski-closed subset of $X$. 
\item 
Suppose that $D_1$ and $D_2$ are ample.  Then any set $R$ of $\Ocal_k$-integral points on $X\setminus D_1\cup D_2$ is contained in a finite union of fibers and sections $C$ with $C\sim C_0$.
Any set $R$ of $\Ocal_k$-integral points on $X\setminus D_1\cup D_2$ is finite if and only if every curve $C\subset X$ over $k$ which is linearly equivalent to $C_0$ or a  fiber $F$, intersects $D_1\cup D_2$ in more than one point.  In particular, this holds if $D_i.F\ge 2$ and $D_i.C_0\ge 2$ for $i=1,2$.
\end{enumerate}
\end{theoremb} 
\section{Preliminaries}
\label{prelim}
\renewcommand{\thetheorem}{\arabic{section}.\arabic{theorem}}
 
We first  introduce some  notation and definitions.

Let $k$ be a number field and let $M_k$ denote the set of inequivalent places of $k$.  Let $D$ be an effective divisor on a nonsingular projective variety $X$, both defined over $k$.  Recall that we can associate to $D$ a height function $h_D$ (well-defined up to $O(1)$) which for points $P\in X(k)\setminus D$ decomposes as a sum of local height functions $h_D=\sum_{v\in M_k}h_{D,v}$.  Let $S$ be a finite set of places of $k$.  A set of points $R\subset X(k)\setminus D$ is called a set of $\Ocal_{k,S}$-integral points on $X\setminus D$ if there exist constants $c_v$, $v\in M_k$, such that $c_v=0$ for all but finitely many $v$, and for all $v\not\in S$,
\begin{equation*}
h_{D,v}(P)\leq c_v
\end{equation*}
for all $P\in R$.  We refer the reader to \cite{Vo} for further details on height functions and integral points.

Let $D$ be a divisor on a nonsingular projective variety $X$, both defined over a field $k$.  For a nonzero rational function $\phi\in k(X)$, we let $\dv(\phi)$ denote the divisor associated to $\phi$.  Then we let $L(D)=\{\phi\in k(X)\mid \dv(\phi)+D\geq 0\}$ and $h^0(D)=\dim H^0(X,\co(D))=\dim L(D)$.  If $h^0(nD)=0$ for all $n>0$ then we let $\kappa(D)=-\infty$.  Otherwise, we define the Kodaira-Iitaka dimension of $D$ to be the integer $\kappa(D)$ such that there exist positive constants $c_1$ and $c_2$ with
\begin{equation*}
c_1 n^{\kappa(D)} \leq h^0(nD)\leq c_2 n^{\kappa(D)}
\end{equation*}
for all sufficiently divisible $n>0$.  We define a divisor $D$ on $X$ to be {\it big} if $\kappa(D)=\dim X$.

\begin{definition}
A divisor $D$ on $X$ is said to be numerically effective, or {\it nef}, if $D\cdot C\ge 0$ for any irreducible curve $C$ on $X$.
\end{definition}

We recall some basic properties of nef divisors on $X$; see \cite{Kle}.

\begin{lemma}
Nef divisors satisfy the following:
\begin{enumerate}[label={\rm \theenumi}] 
\item  Let $n=\dim X$.  If $D_1,\hdots, D_n$ are nef divisors on $X$, then 
$$D_1\cdot D_2\cdots D_n\ge 0.$$
\item  Let $f:X\to Y$ be a morphism and let $D$ be a nef divisor on $Y$ with $f(X)$ not contained in the support of $D$.  Then $f^*(D)$ is nef on $X$.
\end{enumerate}
\end{lemma}
We recall the following generalized Hodge index theorem;  see \cite[Theorem 1.6.1]{Laz}.
\begin{theorem}\label{HI}
Let $D_1,\hdots, D_n$ be nef  divisors on an $n$-dimensional nonsingular projective variety $X$.  Then 
$(D_1\cdot\hdots\cdot D_n)^n\ge D_1^n\cdot \hdots \cdot D_n^n.$
\end{theorem}

When $X$ is a surface and $n=2$, this holds without the nef hypothesis as long as $D_1^2>0$.

For nef divisors, we have the following asymptotic Riemann-Roch formula.

\begin{lemma}\label{big}
Suppose $D$ is a nef divisor on a nonsingular projective variety $X$.  Let $n=\dim X$.
Then $h^0(mD)=\frac{D^n}{n!} \cdot m^n+O(m^{n-1})$. In particular, $D^n>0$ if and only if $D$ is big.
\end{lemma}

The next theorem gives a more refined
bound for the dimension of the space of global sections of nef and big divisors (cf.  \cite{Mat} or \cite[Chapter VI Theorem 2.15.9]{Kol}).

\begin{theorem}[Matsusaka]\label{RR}
Let $X$ be a nonsingular projective variety of dimension $n$ and let $D$ be a nef and big divisor on $X$.  Then
\begin{align*}
h^0(mD)=\frac{D^n}{n!}m^n-\frac{K_X.D^{n-1}}{2(n-1)!}m^{n-1}+O(m^{n-2}).
\end{align*}
\end{theorem}


We will also make use of two basic exact sequences:
 
\begin{lemma}
\label{exact}
Let $D$ be an effective divisor on a nonsingular projective variety $X$ with inclusion map $i:D \to X$.  Let $\mathcal{L}$ be an invertible sheaf on $X$.  Then we have exact sequences
\begin{align*}
&0 \to \mathcal{L}\otimes \co(-D) \to \mathcal{L} \to i_{*}(i^*\mathcal{L}) \to 0,\\
&0 \to H^0(X,\mathcal{L}\otimes \co(-D))\to H^0(X,\mathcal{L}) \to H^0(D,i^*\mathcal{L}).
\end{align*}
\end{lemma}
\begin{proof}
If $D$ is an effective divisor on $X$, then a fundamental exact sequence is
\begin{equation*}
0 \to \co(-D) \to \mathcal{O}_X \to i_{*} \mathcal{O}_D \to 0.
\end{equation*}
Tensoring with $\mathcal{L}$ and using the projection formula, we get the first exact sequence.  Taking global sections then gives the second exact sequence.
\end{proof}

\section{Proof of Theorem \ref{pxn} and Theorem \ref{pxnA}}
\label{CHD}

\renewcommand{\thetheorem}{\arabic{section}.\arabic{theorem}}

We first prove a theorem controlling the Kodaira-Iitaka dimension of certain divisors on blow-ups.

\begin{theorem}\label{bigness}
Let $X$ be a nonsingular projective variety of dimension $n$ over a field $k$ and let $D$ be a nef divisor on $X$.  Let $K_X$ denote the canonical divisor on $X$.
Let $\pi:\tilde X\to X$ be the blow-up along $m$ distinct points $P_1,\hdots,P_m$ of $X$, successively, and let $E_i:=\pi^{-1}(P_i)$ 
be the exceptional divisor for $1\le i\le m$.
\begin{enumerate}[label={\rm \theenumi}]
\item  If $D^n>1$, then $\pi^*D-E_i$ is big for each $1\le i\le m$.\label{biga}
\item  If $D^n=1$ and $K_X.D^{n-1}<1-n$, then $\kappa(\pi^*D-E_i)\geq n-1$ for each $1\le i\le m$.\label{bigb}
\end{enumerate}
\end{theorem}

\begin{remark}
It is easy that part \ref{biga} is sharp.  If $H$ is a hyperplane in $\mathbb{P}^n$ and $P=P_1\in H$, $m=1$, then $H^n=1$ and $\pi^*H-E$ is not big.  Similarly, at least for $n=2$, part \ref{bigb} is also sharp.  Let $P_1,\ldots, P_9\in\mathbb{P}^2$ be nine points in the plane with a unique cubic curve $C$ passing through the nine points.  Let $\pi:X\to\mathbb{P}^2$ be the blow up at $P_1,\ldots, P_8$, with corresponding exceptional divisors $E_1,\ldots, E_8$.  Then $K_X=\pi^*(-3L)+\sum_{i=1}^8E_i$, where $L$ is a line in $\mathbb{P}^2$.  Let $D=\pi^*C-\sum_{i=1}^8E_i\sim -K_X$ be the strict transform of $C$. Then $D^2=1$ and $D.K_X=-1$.  Let $\tilde{\pi}:\tilde X\to X$ be the blow up of $X$ at $P_9$ and let $E$ be the exceptional divisor.  Then $\kappa(\tilde{\pi}^*D-E)=0$.
\end{remark}
\begin{proof}
We first treat the case of blowing-up one point, i.e. $m=1$.  For simplicity of notation, we let
$P=P_1$ and $E=E_1$.
Let $i:E\to \tilde{X}$ be the inclusion map.  From Lemma \ref{exact}, we have an exact sequence
\begin{align*}
0\to H^0(\tilde{X},\co(m\pi^*D-(j+1)E)) \to H^0&(\tilde{X},\co(m\pi^*D-jE))\\
&\to H^0(E,i^*\co(m\pi^*D-jE)).
\end{align*}
Now since $\pi^*D.E=0$, $i^*\co(m\pi^*D-jE)\cong i^*\co(-jE)$.  Recall that we can identify $E$ with $\mathbb{P}^{n-1}$.  Under this identification, $i^*\co(-E)\cong \co_{\mathbb{P}^{n-1}}(1)$.  It follows that
\begin{align*}
\dim H^0(E,i^*\co(m\pi^*D-jE))&=\dim H^0(E,i^*\co(-jE))=\dim H^0(\mathbb{P}^{n-1},\co(j))\\
&=\binom{j+n-1}{n-1}.
\end{align*}
Thus,
\begin{align*}
h^0(m\pi^*D-jE)-h^0(m\pi^*D-(j+1)E)\leq \binom{j+n-1}{n-1}.
\end{align*}
Then
\begin{align*}
h^0(m\pi^*D)-h^0(m\pi^*D-mE)&=\sum_{j=0}^{m-1}h^0(m\pi^*D-jE)-h^0(m\pi^*D-(j+1)E)\\
&\leq \sum_{j=0}^{m-1}\binom{j+n-1}{n-1}=\binom{m+n-1}{n}\\
&\leq \frac{m^n}{n!}+\frac{(n-1)n}{2}\frac{m^{n-1}}{n!}+O(m^{n-2})\\
&\leq \frac{m^n}{n!}+\frac{1}{2}\frac{m^{n-1}}{(n-2)!}+O(m^{n-2}).
\end{align*}
Using Theorem \ref{RR} to compute $h^0(m\pi^*D)$, we find that
\begin{align*}
h^0(m\pi^*D-mE)\geq \frac{D^n}{n!}m^n-\frac{K_X.D^{n-1}}{2(n-1)!}m^{n-1}-\left(\frac{m^n}{n!}+\frac{1}{2}\frac{m^{n-1}}{(n-2)!}\right)+O(m^{n-2}).
\end{align*}
Simplifying, we get
\begin{align*}
h^0(m\pi^*D-mE)\geq \frac{D^n-1}{n!}m^n-\frac{K_X.D^{n-1}+(n-1)}{2(n-1)!}m^{n-1}+O(m^{n-2}),
\end{align*}
proving the theorem for the case of blowing-up one point.

For the general case, we use the following notation.
Let $\tilde\pi_i:X_i\to  X$ be the blow-up at $P_i$,  and let $\epsilon_i:\tilde X\to X_{i}$ be such that $\pi=\tilde\pi_i\circ\epsilon_i$.  
Let $\tilde E_i=\tilde\pi_{i}^{-1}(P_i)$ 
be the exceptional divisor of $\tilde\pi_i$ and note that $E_i=\epsilon_i^*\tilde E_i$.  Then we have the following:
$$
h^0(m(\pi^*D-E_i))=h^0(m(\epsilon_i^*(\tilde\pi_i^* D)-\epsilon_i^*\tilde E_i))\ge h^0(m(\tilde\pi_i^*D- \tilde E_i)).
$$
Thus, the general case follows from the previously proved one-point case applied to the blow-ups $\tilde\pi_i:X_i\to X$, $i=1,\ldots, m$.
\end{proof}

\begin{proof}[Proof of Theorem \ref{pxn} and Theorem \ref{pxnA}]
Since a major part of the proofs of the two theorems are the same, we will largely do the proofs together, giving separate arguments when necessary.  
By Lemma \ref{big}, the assumption that $D_i$ is nef and $D_i^n\geq 1$ for all $i$  
implies that $D_i$ is big for each $1\le i\le n$. We first observe that it suffices to
consider when $\cap_{i=1}^n D_i$ is nonempty.
Indeed, suppose that $\cap_{i=1}^n D_i$ is empty.  Then in the non-Archimedean case, by Theorem \ref {ALW} there exists a proper Zariski-closed subset $Z\subset X$ such that the image of any non-constant analytic map $f:\k\to X\setminus D$ is contained in $Z$.  
While in the arithmetic case, by Theorem \ref {Levin} there exists a proper Zariski-closed subset $Z\subset X$ such that for any set $R$ of $\Ocal_k$-integral points on $X\setminus \cup_{i=1}^nD_i$, the set $R\setminus Z$ is finite.  Therefore, we now assume that $\cap_{i=1}^n D_i$ is not empty.

Since $D_1,\ldots,D_n$ intersect transversally, their intersection contains only points.
Let $\cap_{i=1}^n D_i=\{\mathfrak p_1, \cdots ,\mathfrak p_m\}$.
Let $\pi:\tilde X\to X$ be the blow-up along $\mathfrak p_1,...,\mathfrak p_m$, successively, and let 
$E_j:=\pi^{-1}(\mathfrak p_j)$ be the exceptional divisor for $1\le j\le m$.
For   $1\le i\le n$ and $1\le j\le m$, we let $G_{ij}:=\pi^*(D_i)-E_j=\tilde D_i+E_1+\cdots+E_{j-1}+E_{j+1}+\cdots+E_m$,
where $\tilde D_i$ is the strict transform of $D_i$ under $\pi$.  Clearly, $G_{ij}$ is effective and by Theorem \ref{bigness}, $\kappa(G_{ij})>0$.  It's clear from the construction and our transversality assumptions that $\cap_{\scriptstyle 1\le i\le n\atop\scriptstyle  1\le j\le m}  G_{ij}=\emptyset$ and 
$\cup_{\scriptstyle 1\le i\le n\atop\scriptstyle  1\le j\le m} G_{ij}=\cup_{i=1}^n \pi^{-1}(D_i).$
In the non-Archimedean case, the analytic map $f:\k\to X\setminus \cup_{i=1}^nD_i$ lifts to
$\tilde f:\k\to \tilde X\setminus \cup_{\scriptstyle 1\le i\le n\atop\scriptstyle  1\le j\le m}  G_{ij}$
with $f=\pi\circ\tilde f$.
Theorem \ref{ALW} implies that the image of $\tilde f$ is contained in a proper subvariety of $\tilde X$,
hence the image of $f$ is contained in a proper subvariety of $X$.  If $D_i^n>1$ for all $i$, then $G_{ij}$ is big for all $i$ and $j$ by Theorem \ref{bigness}.  Now applying Theorem \ref{ALW} as before yields the desired result.

In the arithmetic case,  every point in the intersection $\cap_{i=1}^n D_i=\{\mathfrak p_1, \cdots ,\mathfrak p_m\}$ is $k$-rational by assumption.
Then $\tilde X$ and all the $G_{ij}$ are defined over $k$ since  $\pi:\tilde X\to X$ is blown up over $k$-rational points.
Under this construction, any set $R$ of $\Ocal_k$-integral points on $X\setminus \cup_{i=1}^nD_i$ 
lifts to a set $\tilde R$ of $\Ocal_k$-integral points on $\tilde X\setminus \cup_{\scriptstyle 1\le i\le n\atop\scriptstyle  1\le j\le m}  G_{ij}$.
By Theorem \ref {Levin}, $\tilde R$ is contained in a proper Zariski-closed subset $\tilde  Z$ of $\tilde X$, and hence $R$ is contained 
in the proper Zariski-closed subset $\pi (\tilde  Z)$ of $X$.
 If $D_i^n>1$ for all $i$, then $G_{ij}$ is big for all $i$ and $j$ by Theorem \ref{bigness}.  Now applying Theorem \ref{Levin} as before yields the desired result.
\end{proof}

\begin{proof}[Proof of Corollary \ref{cor} and  Corollary \ref{corb}]
By Theorem \ref{HI}, for all $i$,
\begin{align*}
(-K_X.D_i^{n-1})^n\geq (-K_X)^n(D_i^n)^{n-1}\geq (-K_X)^n>(n-1)^n.
\end{align*}
Since $-K_X.D_i^{n-1}\geq 0$ for all $i$, this immediately implies that $K_X.D_i^{n-1}<1-n$ for all $i$.
\end{proof}

\section{Proof of Theorem  \ref{rationalsurface} and Theorem  \ref{rationalsurfaceb}}
 
 We first recall the definition and basic properties of  rational ruled surfaces (cf. \cite[Chapter V.2]{Har} and \cite[Chapter IV.1]{Kol}).

\begin{definition}
Let $k$ be a field and $\bar k$ be the algebraic closure of $k$.
A  {\it rational ruled surface}  is  a surface $X$ (over  $\bar k$), together with a surjective morphism $\pi:X\to {\mathbb P^1}$, such that the fiber $X_y$ is isomorphic to 
${\mathbb P^1}$ for every point $y\in {\mathbb P^1}$.  We say $\pi:X\to {\mathbb P^1}$ is  a {\it rational ruled surface defined over $k$} if both $X$  and the morphism $\pi $ are defined over $k$ and $X$ is a rational ruled surface over $\bar k$.
\end{definition}

\begin{proposition}\label{rationalrule}
A rational ruled surface $X$ (over  $\bar k$) is isomorphic to $X_e:= \mathbf {P}({\mathcal O}_{\mathbb P^1}\oplus{\mathcal O}_{\mathbb P^1}(-e))$ over ${\mathbb P^1}$ for some nonnegative integer $e$.  
Let $F$ denote a fiber on $X$ and let $C_0$ denote a section of $X$ such that $\co(C_0)\cong \co_{ X_e }(1)$. 
Then
\begin{enumerate}
\item[\rm{(a)}]  $\Pic X\cong {\mathbb Z} \oplus {\mathbb Z} $ generated by $C_0\subset X$  and $F$   with $C_0^2=-e$, $F^2=0$, and $C_0\cdot F=1$. 
\item[\rm{(b)}]  Let $K_{X}$ be the canonical divisor on $X$.  Then $K_{X}\sim -2C_0-(2+e)F.$  In particular, $K_{X}^2=8$.
\item[\rm{(c)}]  Let $D$ be a divisor on $X$ equivalent to $aC_0+bF$ in $\Pic X$.  Then 
\begin{enumerate}
\item[\rm(i)] If $D$ is an irreducible curve $\not\sim C_0,\, F$, then $a,b>0$, $b\ge ae$, and $D^2>0$.
\item[\rm(ii)] $D$ is big if and only if $a>0$ and $b>0$.
\item[\rm(iii)] $D$ is ample if and only if $a>0$ and $b>ae$.
\end{enumerate}
\end{enumerate}
\end{proposition}
\begin{remark}
When $e=0$, $X_0$ is isomorphic to ${\mathbb P^1}\times {\mathbb P^1}$; when $e=1$, $X_e=X_1$ is $\mathbb{P}^2$ blown up at one point.  
\end{remark}

 \begin{proof}[Proof of Theorem \ref{rationalsurface}]
A rational ruled surface $X$ is isomorphic to $X_e$ with $e\ge 0$  as described in Proposition \ref{rationalrule}.    Let $D_1$ and $D_2$ be big effective divisors on $X$.  Suppose first that $D_1$ and $D_2$ have irreducible components $E_1$ and $E_2$, respectively, with $E_i\not\sim C_0,F$, for $i=1,2$.  Then by Proposition \ref{rationalrule}, $E_i^2\geq 1$ for $i=1,2$.  Let $K_{X}$ be the canonical divisor on $X$.  By the Hodge index theorem (or direct calculation), $(K_{X}.E_i)^2\geq K_X^2 E_i^2\geq K_X^2=8$ and $K_{X}.E_i\leq -3$.   Then by Theorem \ref{pxn},  the image of an analytic map $f:\k\to X\setminus (D_1\cup D_2)\subset X\setminus (E_1\cup E_2)$ is contained in a proper subvariety of $X$.  
  Suppose now that, say, $D_1$ has every irreducible component linearly equivalent to either $C_0$ or $F$.  Since $D_1$ is big, by Proposition \ref{rationalrule}, $D_1$ must contain at least two irreducible components $C$ and $F'$ with $C\sim C_0$ and $F'$ a fiber.  Since $D_2$ is linearly equivalent to a positive integral linear combination of $C$ and $F'$,  there exists a non-constant function $\phi\in \k(X)^*$ with poles and zeros only in the support of $D_1$ and $D_2$.  Consider an analytic map $f:\k\to X\setminus (D_1\cup D_2)$.  Then $\phi\circ f:\k\to \mathbb{A}^1\setminus \{0\}$ is analytic, and hence constant.  It follows that the image of $f$ is contained in a proper subvariety of $X$.

Assume furthermore that $D_1$ and $D_2$ are ample.  By (a), the image of a non-constant analytic map $f:\k\to X\setminus (D_1\cup D_2)$ is contained in a curve in $X$.   Let $C$ be the Zariski closure of the image of $f$ in $X$.  
If $C\cap (D_1\cup D_2)$ contains more than one point, then $f$ must be constant.  On the other hand, $D_i\cap C\ne\emptyset$ for $i=1,2$, since $D_1$ and $D_2$ are ample.
Therefore, we only need to consider when $C\cap  D_1=C\cap  D_2 =\{x\} $ for some $x\in X$. 

Let $\pi: \tilde X\to X$ be the blow-up of $X$ at $x$ with exceptional divisor $E$.   Let $\tilde C$ be the strict transform of $C$ and $\tilde D_i$ the strict transform of $D_i$, $i=1,2$.  Then $f:\k\to X\setminus D_1\cup D_2$ lifts to
$\tilde f:\k\to \tilde X\setminus \tilde D_1\cup\tilde D_2$
with $f=\pi\circ\tilde f$ and the image of $\tilde f$ is contained in $\tilde C$.
Denote by $m=m_x( C)$ the multiplicity of $C$ at $x$.  If $(C.D_i)_x>m=m_x( C)\cdot m_x(D_i)$  for $i=1,2$,
then each $\tilde{D}_i$ must intersect $\tilde C$ at some point on $\tilde X$ lying above $x$.  Since $D_1$ and $D_2$ intersect transversally, $\cap_{i=1}^2 \tilde D_i\cap E=\emptyset$.  Thus, there must be at least two points on $\tilde C$ lying above $x$.  Consequently,  $\tilde f:\k\to \tilde X\setminus \tilde D_1\cup\tilde D_2$ is constant by Theorem \ref{Pic} and hence $f$ is also constant.
Therefore, it remains to consider when $(C.D_i)_x=m$ for $i=1$ or $i=2$.  

Without loss of generality, let $(C.D_1)_x=m$. 
Suppose that $C$ is linearly equivalent to $cC_0+dF$ and $D_1$ is linearly equivalent to $a C_0+b F$.
 Proposition \ref{rationalrule} implies that the intersection multiplicity is given by
 $$
 (C.D_1)_x=C.D_1=a d+c(b -a e).
 $$
 
Assume that $C$ is not linearly equivalent to $C_0$ or $F$.  From Proposition \ref{rationalrule}, we have that  $a,b-ae,c, d>0$.
By taking $F$ to be the fiber passing through $x$, we see that  $c=C.F\ge m $.
Then 
\begin{align*}
m=C.D_1=a d+c(b -a e)>c\geq m,
\end{align*}
a contradiction.  This proves the first statement in \ref{rsb}.  The second statement follows from Theorem \ref{Pic} and the observation that if some fiber or some section $C\sim C_0$ intersects $D_1\cup D_2$ in exactly one point, then since every fiber and every section is isomorphic to $\mathbb{P}^1$ and $\mathbb{A}^1$ admits a non-constant analytic map, there exists a non-constant analytic map from $\k$ to $X\setminus D_1\cup D_2$.

For the last statement, note that any fiber or section must intersect one of $D_1$ or $D_2$ transversally at $x$, as any fiber or section is nonsingular and $D_1$ and $D_2$ intersect transversally.  If $D_i.F\ge 2$ and $D_i.C_0\ge 2$ for $i=1,2$, then this implies that every fiber and every section intersects $D_1\cup D_2$ in more than one point.
\end{proof}
 \begin{proof}[Proof of Theorem \ref{rationalsurfaceb}]
We will use the same notation as in the proof of Theorem \ref{rationalsurface} and only give arguments that are different from the non-Archimedean case.  We note that the section $C_0$ may not be defined over $k$ (and consequently, we work with linear equivalence over $\bar{k}$ below).  Let $D_1$ and $D_2$ be big effective divisors on $X$.  Suppose first that $D_1$ and $D_2$ have irreducible components $E_1$ and $E_2$, respectively, with $E_i\not\sim C_0,F$, for $i=1,2$.  We have shown that  $E_i^2\geq 1$ for $i=1,2$  and $K_{X}.E_i\leq -3$.  
  Then by Theorem \ref{pxnA}, any set $R$ of $\Ocal_k$-integral points on $X\setminus (D_1\cup D_2)\subset X\setminus (E_1\cup E_2)$ is contained in a proper Zariski-closed subset of $X$.    Suppose now that, say, $D_1$ has every irreducible component linearly equivalent to either $C_0$ or $F$.  Since $D_1$ is big, by Proposition \ref{rationalrule}, $D_1$ must contain at least two irreducible components $C$ and $F'$ with $C\sim C_0$ and $F'$ a fiber.  Since $D_2$ is linearly equivalent to a positive integral linear combination of $C$ and $F'$, there exists a non-constant function $\phi\in \bar \QQ(X)^*$ with poles and zeros only in the support of $D_1$ and $D_2$.  
Since all the irreducible components of $D_1$ and $D_2$ are defined over $k$, we may take $\phi$ to be defined over $k$.  Then $\phi(R)$ is a set of  $\Ocal_k$-integral points on $ \mathbb{A}^1\setminus \{0\}$  which is finite.  Consequently, we may write
$\phi(R)=\{\alpha_1,\hdots,\alpha_m\}\subset k\setminus \{0\}$.   It follows that the image of $R$ is contained in a proper subvariety of $X$  given by  $\cup_{i=1}^m\{\frak p\in X| \, \phi(\frak p)=\alpha_i\}$.

Assume furthermore that $D_1$ and $D_2$ are ample.  
Suppose that $R$ is an infinite set of $\Ocal_k$-integral points on $X\setminus (D_1\cup D_2)$.
By (a), one can find a Zariski closed curve $C$  containing an infinite subset $R'$ of $R$.
Then $C$ must be defined over $k$ and $C\cap (D_1\cup D_2)$ contains at most two points by Siegel's theorem.
If $C\cap (D_1\cup D_2)$ contains exactly two points $\frak p$, $\frak q$, then either $\frak p$ and $\frak q$ are $k$-rational points or they are conjugate to each other.  Since $C$, $D_1$ and $D_2$ are defined over $k$, the second case implies that 
$C\cap D_1=C\cap D_2=\{\frak p, \frak q\}$ and hence $\{\frak p, \frak q\}\subset D_1\cap D_2$ contradicting our assumption that $D_1\cap D_2$ contains only $k$-rational points.  The remaining case that $\frak p$ and $\frak q$ are both $k$-rational points is impossible due to Theorem~\ref{Levin}\ref{Levinc}.  Therefore, $C\cap (D_1\cup D_2)$ contains at most one point.
On the other hand, $D_i\cap C\ne\emptyset$ for $i=1,2$, since $D_1$ and $D_2$ are ample.
Therefore, we only need to consider when $C\cap  D_1=C\cap  D_2 =\{x\} $ for some $x\in X$. 
Moreover,  $x$ is a $k$-rational point since it is a point in $  D_1 \cap  D_2$. 

Let $\pi: \tilde X\to X$ be the blow-up of $X$ at $x$ with exceptional divisor $E$.   Let $\tilde C$ be the strict transform of $C$ and $\tilde D_i$ the strict transform of $D_i$, $i=1,2$.  
We note that  $\tilde X$, $\tilde C$, and $\tilde D_i$ are all defined over $k$ since $x$ is $k$-rational.
Then this infinite set $R'$ of $\Ocal_k$-integral points on $X\setminus (D_1\cup D_2)$ lifts
to an infinite set $\tilde R'$ of $\Ocal_k$-integral points on $\tilde X\setminus \tilde D_1\cup\tilde D_2$ with $R'=\pi(\tilde R')$.  
Denote by $m=m_x( C)$ the multiplicity of $C$ at $x$.  If $(C.D_i)_x>m=m_x( C)\cdot m_x(D_i)$  for $i=1,2$,
then each $\tilde{D}_i$ must intersect $\tilde C$ at some point on $\tilde X$ lying above $x$.  Since $D_1$ and $D_2$ intersect transversally, $\cap_{i=1}^2 \tilde D_i\cap E=\emptyset$.  Thus, there must be at least two points on $\tilde C$ lying above $x$.  
As before, by Siegel's theorem, it suffices to consider when there are exactly two points in $\tilde C\cap  (\tilde D_1\cup\tilde D_2)$.
The above construction shows that this can only happen if $\tilde C\cap  \tilde D_1=\{\frak p\}$, $\tilde  C\cap  \tilde D_2=\{\frak q\}$, and $\frak p\ne\frak q$.  Since $\tilde C$ and $\tilde D_i$, $i=1,2$, are defined over $k$, $\frak p$ and $\frak q$ must be $k$-rational points.  However, Theorem  \ref{Levin}\ref{Levinc}  implies that any set of  $\Ocal_k$-integral points on $\tilde C\setminus (\tilde D_1\cup\tilde D_2)$ is finite which yields a contradiction.
Therefore, it remains to consider when $(C.D_i)_x=m$ for $i=1$ or $i=2$.  
The same argument as in the proof of Theorem \ref{rationalsurface} shows that $C$ is linearly equivalent to $C_0$ or $F$ under this assumption.
 This proves the first statement in \ref{rsb}.  

For the second statement, we continue our argument from above and consider when $C$ is linearly equivalent to $C_0$ or $F$.  Suppose that $C$ intersects 
$D_1\cup D_2$ in at least two points.  If there are two $k$-rational points or more than two points in $C\cap \{D_1\cup D_2\}$,
then we get a contradiction again by Theorem  \ref{Levin}\ref{Levinc} or by Siegel's theorem.
The only case left is $C\cap \{D_1\cup D_2\}=\{\frak p,\frak q\}$ and one of the points is not $k$-rational.
Since $C$ and the $D_i$ are defined over $k$, this implies that $\frak{p}$ and $\frak{q}$ are conjugate over $k$ and $C\cap D_1=C\cap D_2=\{\frak p,\frak q\}$.  Then $\{\frak p,\frak q\}\subset D_1\cap D_2$, which contradicts our assumption that the intersection points of $D_1$ and $D_2$ are $k$-rational.  Therefore, the cardinality of $R$ must be finite.
For the converse direction of the second statement, we first observe that if  an irreducible curve $C$ defined over $k$  intersects $D_1\cup D_2$ in exactly one point, then this point must be $k$-rational.  Furthermore, if $C$ is linearly equivalent to $F$ or $C_0$ then $C$ is a fiber or a section which is isomorphic to $\mathbb{P}^1$.  Indeed, $C$ is $k$-isomorphic to $\mathbb{P}^1$ since $C$ is defined over $k$ and $C(k)$ is not empty.  Since
 $\mathbb{A}^1$ admits infinitely many $\Ocal_k$-integral points, there exists infinitely many $\Ocal_k$-integral points on $C\setminus (D_1\cup D_2)\subset X\setminus (D_1\cup D_2)$. 
 
For the last statement, the proof is the same as in Theorem \ref{rationalsurface}.
\end{proof}

\end{document}